\documentclass{amsart}
\usepackage[foot]{amsaddr}
\usepackage{lmodern}
\usepackage[T1]{fontenc}
\usepackage[utf8]{inputenc}

\usepackage{todonotes}
\usepackage{xcolor}
\usepackage{overpic}

\usepackage{amsmath,amssymb,amsfonts,amsthm}
\usepackage{mathtools}

\usepackage{hyperref}
\hypersetup{
  pdfauthor={Johanna Lercher, Daniel Scharler, Hans-Peter Schröcker, Johannes Siegele},
  pdftitle={Factorization of Quaternionic Polynomials of Bi-Degree (n,1)},
  pdfkeywords={},
  hidelinks,
}

\newcommand{\qi}{\mathbf{i}}
\newcommand{\qj}{\mathbf{j}}
\newcommand{\qk}{\mathbf{k}}
\newcommand{\Cj}[1]{{#1}^\ast}
\newcommand{\No}[1]{#1 {#1}^\ast}

\newcommand{\R}{\mathbb{R}}
\newcommand{\C}{\mathbb{C}}

\newcommand{\Nat}{\mathbb{N}}
\newcommand{\N}{\mathcal{N}}
\renewcommand{\P}{\mathbb{P}}

\newcommand{\Q}{\mathbb{H}}
\newcommand{\Qt}{\mathbb{H}[t]}
\newcommand{\Qs}{\mathbb{H}[s]}
\newcommand{\Qts}{\mathbb{H}[t,s]}
\newcommand{\Rt}{\mathbb{R}[t]}
\newcommand{\Rs}{\mathbb{R}[s]}
\newcommand{\Rts}{\mathbb{R}[t,s]}
\newcommand{\Hmn}{\mathbb{H}_{mn}}
\newcommand{\Hstar}[1]{\mathbb{H}_{\ast #1}}

\newcommand{\ov}[1]{\hat{#1}}
\newcommand{\Cq}{\mathbb{C}\mathbb{H}}

\newcommand{\fact}{\mathrm{fact}}
\newcommand{\Fact}{\mathrm{Fact}}
\newcommand{\F}[5]{(F_{#1_i,#2,#3_j})_{#4,#5}}
\newcommand{\rul}{[Q_0(z)] \vee [Q_1(z)]}
\newcommand{\m}[2]{\operatorname{mult}(#1,#2)}

\newtheorem{thm}{Theorem}[section]

\newtheorem{prop}[thm]{Proposition}
\newtheorem{cor}[thm]{Corollary}
\theoremstyle{definition}
\newtheorem{defn}[thm]{Definition}

\theoremstyle{remark}
\newtheorem{rmk}[thm]{Remark}
\newtheorem{example}[thm]{Example}

\title[]{Factorization of Quaternionic Polynomials of Bi-Degree (n,1)}

\date{\today}

\author{J. Lercher \and D. F. Scharler \and H.-P. Schröcker \and J. Siegele}
\address{Department of Basic Sciences in Engineering Sciences, University of Innsbruck, Technikerstr.~13, 6020 Innsbruck, Austria}
\email{johanna.lercher@uibk.ac.at}
\email{daniel.scharler@uibk.ac.at}
\email{hans-peter.schroecker@uibk.ac.at}
\email{johannes.siegele@uibk.ac.at}

\subjclass[2020]{16S36, 12D05} 

\keywords{left/right factor, factorization algorithm, spherical kinematics, ruled surface}

\begin{document}

\begin{abstract}
  We consider polynomials of bi-degree $(n,1)$ over the skew field of
  quaternions where the indeterminates commute with each other and with all
  coefficients. Polynomials of this type do not generally admit factorizations.
  We recall a necessary and sufficient condition for existence of a
  factorization with univariate linear factors that has originally been stated
  by Skopenkov and Krasauskas. Such a factorization is, in general, non-unique
  by known factorization results for univariate quaternionic polynomials. We
  unveil existence of bivariate polynomials with non-unique factorizations that
  cannot be explained in this way and characterize them geometrically and
  algebraically. Existence of factorizations is related to the existence of
  special rulings of two different types (left/right) on the ruled surface
  parameterized by the bivariate polynomial in the projective space over the
  quaternions. Special non-uniqueness in above sense can be explained
  algebraically by commutation properties of factors in suitable factorizations.
  A necessary geometric condition for this to happen is degeneration to a point
  of at least one of the left/right rulings.
\end{abstract}

\maketitle

\section{Introduction}
\label{sec:introduction}

Let $\Q$ be the skew field of \emph{real quaternions}. Factorization theory for
univariate quaternionic polynomials $Q \in \Qt$ has been developed in
\cite{niven41,gordon65} and received recent attention because of its relation to
kinematics and mechanism science \cite{hegedus13,hegedus15}. A \emph{fundamental
  theorem of algebra} also holds true in the quaternionic case. Consequently,
each univariate quaternionic polynomial admits a factorization into linear
factors (c.~f. \cite{niven41,gordon65,Gentili08} for quaternionic polynomials
and \cite{hegedus13,li18,li19} for similar results in more general associative
real algebras). Due to the non-commutativity of the division ring $\Q$,
factorizations into linear factors need not be unique. A quaternionic polynomial
$Q \in \Qt$ of degree $n$ admits, in general, $n!$ different factorizations with
linear factors.

Not much is known about factorization theory of multivariate quaternionic
polynomials. The only publication we are aware of is the recent paper
\cite{Skopenkov19} by Skopenkov and Krasauskas. Strangely enough, they derive
interesting and innovative factorization results for bivariate quaternionic
polynomials as an auxiliary tool for a seemingly unrelated topic, the
classification of surfaces in Euclidean three space that are foliated by two
families of circles. In this article, we build upon the results and ideas of
Skopenkov and Krasauskas. One of the insights is that the bivariate case, where
factorizations are exceptional, is much harder than the univariate case, where
factorizations always exist.

Already the case of factorization of bivariate polynomials of degree one in one
of the indeterminates into univariate linear factors is interesting and will be
in the focus of this paper. Motivated by potential applications in kinematics
(c.\,f. Section~\ref{sec:kinematics}), we assume that indeterminates and
coefficients commute. Therefore, it is sufficient to consider polynomials $Q \in
\Qts$ of bi-degree $(n,1)$. Denote by $\Cj{Q}$ the polynomial obtained by
conjugating the coefficients of $Q$. A simple necessary condition for existence
of a factorization of the shape
\begin{equation}
\label{Hs1factors}
	Q=a(t-h_1)\cdots(t-h_n)(s-h)(t-k_1)\cdots(t-k_m)
\end{equation}
with quaternions $a$, $h_1$, \ldots, $h_n$, $h$, $k_1$, \ldots, $k_m \in \Q$ is
existence of real polynomials $P \in \Rt$ and $R \in \Rs$ such that $\No{Q} =
PR$. By \cite[Lemma~2.9]{Skopenkov19}, this condition is also sufficient.

In Section~\ref{sec:preliminaries} we settle our notation and recall some basic
facts on quaternionic polynomials and their factorization theory. At the
beginning of Section~\ref{sec:star-one-polynomials} we recall the original
result of Skopenkov and Krasauskas and introduce the notion of equivalent
factorizations: The factorization of the univariate polynomials
$(t-h_1)\cdots(t-h_n)$ and $(t-k_1)\cdots(t-k_m)$ is in general not unique and
this implies non-uniqueness of factorization \eqref{Hs1factors}. We consider two
factorizations of $Q$ as equivalent if they arise in this way.

Our first substantial new contribution is a geometric interpretation of
factorizability in terms of the ruled surface parameterized by $Q$ in the
complex extension of $\P(\Q) = \P^3(\R)$ in Section~\ref{sec:null-lines}. Linear
factors $t-h_i$ or $t-k_j$ correspond to ``null lines'', that is rulings of the
``null quadric'' $\N$ given by the quadratic form $q \mapsto q\Cj{q}$. In
general, factors to the left and to the right of the $s$-factor $s-h$ are
distinguished by the two families of rulings on~$\N$.

There exist bivariate polynomials of bi-degree $(n,1)$ with non-equivalent
factorizations. If $Q$ is such a polynomial, we will show that one is always
able to compute a factorization of shape $\eqref{Hs1factors}$ where either
$t-h_n$ or $t-k_1$ commutes with the $s$-factor $s-h$. The respective $t$-factor
can then be viewed as a factor that may appear on the left \emph{or} on the
right of $s-h$. Therefore, the corresponding ruled surface contains at least one
null line that can be considered as an element of both families of rulings
on~$\N$. This is only possible if the respective null line degenerates to a
point. The converse of this statement is, however, not true. A detailed
investigation of these cases is on the agenda in Sections~\ref{sec:uniqueness}
and~\ref{sec:non-uniqueness}.

While factorization of bivariate polynomials is an interesting topic in its own
right, our research is also motivated by applications in kinematics. We briefly
illustrate the underlying ideas in Section~\ref{sec:kinematics} and we also
explain why factorization results for polynomials of bi-degree $(n,1)$ are too
limiting to allow the direct transfer of kinematic constructions from
\cite{hegedus13}. This is no longer the case for polynomials of higher bi-degree
and the theory presented in this article is of fundamental importance in their
factorization theory (c.\,f.~\cite{lercher21}).

\section{Preliminaries}
\label{sec:preliminaries}

We denote the algebra of \emph{real quaternions} by $\Q$. It is generated by the quaternion units $1$, $\qi$, $\qj$ and $\qk$ over the real numbers $\R$.
An element $h \in \Q$ is of the form
\begin{equation*}
	h = h_0+h_1 \qi + h_2 \qj + h_3 \qk \quad \text{with} \quad h_0, h_1, h_2, h_3 \in \R.
\end{equation*}
The relations
\begin{equation}
\label{mult}
	\qi^2=\qj^2=\qk^2=\qi\qj\qk=-1
\end{equation}
completely define the multiplication in~$\Q$. The \emph{conjugate} of $h$ is
$\Cj{h} \coloneqq h_0 - h_1 \qi - h_2 \qj - h_3 \qk$, its \emph{norm} is given by
\begin{equation*}
\No{h}=h_0^2+h_1^2+h_2^2+h_3^2.
\end{equation*}

Multiplication of quaternions is not commutative but $\Q$ is at least a division
ring: If $h \neq 0$, it is invertible and its \emph{inverse} is given by
$h^{-1}=\Cj{h}/(\No{h})$. Conjugation of quaternions is anticommutative, hence
for $h,k \in \Q$ one obtains $\Cj{(hk)}=\Cj{k}\Cj{h}$.

By $\Cq=\Q+\mathrm{i}\Q$ we denote the algebra of \emph{complex quaternions}. It
is the $\C$-algebra generated by $1$, $\qi$, $\qj$ and $\qk$ with the
multiplication rules \eqref{mult}. Note that $\mathrm{i} \in \C$ denotes the
imaginary unit of complex numbers. It has to be distinguished from the
quaternion unit $\qi$. A complex quaternion is of the form
\begin{equation*}
	h = h_0+h_1 \qi + h_2 \qj + h_3 \qk \quad \text{with} \quad h_0, h_1, h_2, h_3 \in \C.
\end{equation*}
We define the \emph{(complex) conjugate} of $h$ by $\Cj{h} \coloneqq h_0 - h_1
\qi - h_2 \qj - h_3 \qk$ and its norm by $\No{h}=h_0^2+h_1^2+h_2^2+h_3^2$. The
algebra of complex quaternions $\Cq$ contains zero divisors. These are precisely
the elements $h \in \Cq\setminus \{0\}$ with vanishing norm $\No{h} = 0$. Any
complex quaternion $h \in \Cq\setminus \{0\}$ with $\No{h} \neq 0$ is invertible
and $h^{-1}=\Cj{h}/(\No{h})$. Sometimes we view $\Q$ (or $\Cq$) as a real (or
complex) vector space of dimension four and we also consider the real projective
space $\P(\Q) = \P^3(\R)$ (or the complex projective space $\P(\Cq) =
\P^3(\C)$). Projective points are denoted by square brackets: We write $[q]$ for
a projective quaternion represented by $q \neq 0$. For $p$, $q \neq 0$ the
symbol $[p] \vee [q]$ denotes the connecting line of $[p]$ and $[q]$ if these
points are different. If $[p] = [q]$, we have $[p] \vee [q] = [p] = [q]$.

By $\Qts$ we denote the set of bivariate polynomials with coefficients in $\Q$.
Addition and scalar multiplication are defined as in the commutative case.
Polynomial multiplication is defined by the requirement that $t$ and $s$ commute
with the coefficients and with each other. This convention is motivated by
potential later applications in kinematics and mechanism science where bivariate
polynomials parameterize two-parametric rational motions
(c.\,f.~Section~\ref{sec:kinematics}). In this context, $t$ and $s$ serve as
\emph{real} motion parameters and the real numbers form the center of~$\Q$.

For
\begin{equation*}
	Q \coloneqq \sum_{\alpha \coloneqq (\alpha_1,\alpha_2) \in \Nat_0 \times \Nat_0}q_{\alpha}t^{\alpha_1}s^{\alpha_2} \in \Qts
\end{equation*}
the \emph{conjugate polynomial} of $Q$ is defined by conjugating its coefficients:
\begin{equation*}
	\Cj{Q} \coloneqq \sum_{\alpha \coloneqq (\alpha_1,\alpha_2) \in \Nat_0 \times \Nat_0}\Cj{q}_{\alpha}t^{\alpha_1}s^{\alpha_2} \in \Qts.
\end{equation*}
The \emph{norm polynomial} $\No{Q}$ is a real bivariate polynomial.\\

In order to state our results in a clear and simple way, we often consider monic polynomials without real polynomial factors of positive degree:

\begin{itemize}
\item We call a polynomial $Q \in \Qt[s]$ \emph{monic} if its leading
  coefficient is a monic polynomial. Given a non-monic polynomial $Q =
  \sum_{i=0}^{n}Q_is^{i} \in \Qt[s]$, it suffices to consider existence and
  (non-)uniqueness of factorizations of the monic polynomial $Q'=a^{-1}Q$ where
  $a$ denotes the leading coefficient of $Q_n$.
\item Given $Q = Q_0 + Q_1\qi + Q_2\qj + Q_3\qk \in \Qts$ we denote by
  $\gcd(Q)$ the monic greatest common divisor of the polynomials $Q_0$, $Q_1$,
  $Q_2$, $Q_3 \in \Rt[s]$. There exists a unique polynomial $Q'$ with $Q =
  \gcd(Q)Q'$. Obviously, $Q'$
  has no non-constant real polynomial factor.
\end{itemize}

We use the following notations: For $Q \in \Qts$ we define $\deg_tQ$ as the
degree of $Q$ viewed as an element of $\Q[s][t]$. Analogously, $\deg_sQ$ is
the degree of $Q$ viewed as an element of $\Qt[s]$. Following Skopenkov and
Krasauskas \cite{Skopenkov19}, we define
\begin{equation*}
	\Hmn \coloneqq \{Q \in \Qts\colon \deg_tQ \leq m \text{ and } \deg_sQ \leq n\}
\end{equation*}
and
\begin{equation*}
	\Hstar{n} \coloneqq \bigcup_{m \in \Nat_0} \Hmn.
\end{equation*}
In this article, we study factorizations of bivariate polynomials of the form
\begin{equation*}
	Q=Q_0+Q_1s \in \Hstar{1}
\end{equation*}
with $Q_0, Q_1 \in \Qt$. We call them \emph{star-one-polynomials}.

\subsection{Factorization of Univariate Polynomials}
\label{sec:univariate}
Factorizability of univariate quaternionic polynomials will turn out to be of great importance for our theory. For later reference, we formulate a theorem that states some crucial univariate factorization results. We provide short sketches of proofs even though the results are already well known. For more details we refer to the respective literature.

\begin{thm}
\label{thm:univ}
	Let $Q \in \Qt$ be a quaternionic polynomial, $\deg(Q)\geq 1$. There is a unique representation $Q=\gcd(Q)Q'$ with a polynomial $Q' \in \Qt$ and $\deg(Q')=n$, $n \in \mathbb{N}_0$.
	
		\begin{enumerate}
		\item[(a)] If $\gcd(Q) \neq 1$, the real polynomial
      $\gcd(Q)$ admits a factorization into univariate linear factors over~$\Q$.
		\item[(b)] There exist up to $n$ different monic linear
      right factors of the polynomial~$Q'$.\footnote{We call a factor $t-h \in
        \Q[t]$ a linear \emph{right factor} of a univariate polynomial $Q \in
        \Q[t]$ if $Q=S(t-h)$ for an appropriate $S \in \Q[t]$. Beware that this
        common denotation in the context of non-commutative rings will change
        its meaning in the context of star-one-polynomials (c.~f. Section
        \ref{sec:null-lines}).}
		\item[(c)] There exist up to $n!$ different factorizations of $Q'$ into
      univariate linear factors over $\Q$.
	\end{enumerate}
\end{thm}

\begin{proof}[Sketch of proof]
The polynomial $\gcd(Q)$ admits a factorization into linear factors over $\C$. All
  factorizations of $\gcd(Q)$ with linear factors over $\Q$ are found by replacing the
  complex unit $\mathrm{i}$ in each pair of conjugate complex linear factors by
  a quaternion $h$ satisfying $h^2 = -1$. These are precisely the quaternions
  with zero real part and $\No{h} = 1$ (c.\,f. \cite{huang02}). This shows Part~(a).

Since $Q'$ does not have a real polynomial factor of positive degree, the norm polynomial can be written as $\No{Q'}=\No{a} M_1\cdots M_n$ with $a \in \Q$ and monic, quadratic, irreducible real polynomials $M_1, \ldots, M_n$ (the case $n=0$ is also possible). By Part~(a), each $M_i$ can be
  factored in infinitely many ways as $M_i = (t-\Cj{h_i})(t-h_i)$, $h_i \in \Q$. There is
  precisely one linear factor $t-h_i$ that is also a right factor of $Q'$ (c.\,f.~\cite{niven41,hegedus13}). Using
  polynomial division we find $T$, $U \in \Q[t]$ such that $Q' = TM_i + U$ and $U
  = u_1t + u_0 \in \Q[t]$ whence $h_i = -u_1^{-1}u_0$. If $M_1, \ldots, M_n$ are
  pairwise different, we obtain $n$ different right factors of $Q'$. This
  construction is also necessary so that the claim of Part~(b) follows.

 In order to find a factorization of $Q'$ into univariate linear factors, we use Part~(b) and iteratively produce linear right factors of $Q'$. The thus obtained factorization depends on an order of the
factors $M_1$, \ldots, $M_n$. Using different orders one obtains all
factorizations of $Q'$. In general (if the factors $M_1$, \ldots, $M_n$ are
pairwise different), there are $n!$ factorizations (c.\,f.~\cite{hegedus13}), which proves Part~(c).	
\end{proof}

\subsection{A Necessary Condition for Bivariate Factorization}
\label{sec:necessary}

In contrast to the univariate case, a generic bivariate polynomial $Q \in \Qts$
does not admit a factorization with univariate linear factors.
\begin{defn}
\label{nec}
Let $Q \in \Qts$. We say that $Q$ satisfies the \emph{necessary factorization condition} if $\No{Q}=PR$ for $P \in \Rt$ and $R \in \Rs$.
\end{defn}
In order to see that this condition is really necessary for existence of a
factorization with univariate linear factors, assume that $Q$ can be written as
\begin{equation*}
  Q = a(u_1-h_1)\cdots(u_n-h_n)
\end{equation*}
with $u_i \in \{t, s\}$ and $a$, $h_i \in \Q$ for $i \in \{1,\ldots,n\}$. The norm polynomial is
then
\begin{equation*}
  \begin{aligned}
  \No{Q} &= a(u_1-h_1)\cdots(u_n-h_n) \Cj{(u_n-h_n)}\cdots\Cj{(u_1-h_1)}\Cj{a} \\
         &= a\Cj{a} (u_1-h_1)\Cj{(u_1-h_1)} \cdots (u_n-h_n)\Cj{(u_n-h_n)}.
  \end{aligned}
\end{equation*}
Here we repeatedly used the fact that $(u_i-h_i)\Cj{(u_i-h_i)}$ is a real
polynomial and commutes with all other polynomials. Moreover,
$(u_i-h_i)\Cj{(u_i-h_i)}$ is in $\Rt$ or $\Rs$ from which the claimed
factorization $\No{Q} = PR$ of the necessary factorization condition follows.

\section{Factorization of Star-One-Polynomials}
\label{sec:star-one-polynomials}

Inspired by Skopenkov and Krasauskas \cite{Skopenkov19} and motivated by
potential applications in kinematics, we consider polynomials $Q \in \Hstar{1}$
which admit a factorization into \emph{univariate linear factors}. An important
result, which is more or less stated in \cite{Skopenkov19}, is that the
necessary factorization condition of Definition~\ref{nec} is also sufficient for
polynomials in $\Hstar{1}$. It is not sufficient for arbitrary polynomials in
$\Qts$ by \cite[Example~1.5]{Skopenkov19} (taken from \cite{beauregard93}). If a
factorization exists, we will show that it is in general essentially unique
(that is, unique up to obvious ambiguities arising from different factorizations
of univariate polynomials). Let us continue by stating an important result of
\cite{Skopenkov19}.

\begin{thm}[{\cite[Lemma 2.9]{Skopenkov19}}]
\label{SK2}
If polynomials $Q \in \Hstar{1}$ and $P$, $R \in \Rts$ satisfy $\No{Q}=PR$, then
either $(P,Q,R)$ or $(R,Q,P)$ equals $(\No{(AC)}D,$ $ABCD,$ $\No{B}D)$ for some
$A$, $C \in \Qt$, $B \in \Qts$, $D \in \Rts$.
\end{thm}

The proof of Theorem~\ref{SK2} in \cite{Skopenkov19} is constructive and can be
cast into an algorithm. As a corollary to Theorem~\ref{SK2} we will prove that
the necessary factorization condition is also sufficient for polynomials $Q \in
\Hstar{1}$.

\begin{cor}
\label{univfact}
For $Q \in \Hstar{1}$ the necessary factorization condition is also sufficient.
\end{cor}

\begin{proof}
As shown in Section \ref{sec:preliminaries}, it is no loss of generality to
assume that $Q$ is monic with $\gcd(Q)=1$. Otherwise, we find a representation
$Q=\gcd(Q)Q'$ with $Q' \in \Hstar{1}$. Validity of the necessary factorization
condition guarantees that the real polynomial $\gcd(Q)$ can be decomposed into
univariate polynomials that admit factorizations over $\Q$ by Theorem
\ref{thm:univ}, Part~(a). Therefore, this factor is negligible
and we only have to consider polynomials without real polynomial factors of
positive degree.

Since $\No{Q} = PR$ with $P \in \Rt$ and $R \in \Rs$ we can apply
Theorem~\ref{SK2} and obtain
\begin{equation*}
	Q=ABCD \quad \text{with} \quad P=\No{(AC)}D \quad \text{and} \quad R = \No{B}D.
\end{equation*}
By the assumption $\gcd(Q) = 1$, $D$ has to be constant. Since $R \in \Rs$, we
obtain $B \in \Qs$. Moreover, $\deg B \leq 1$ because $Q \in \Hstar{1}$. Hence
$Q=ABC$ with univariate factors $A, C \in \Qt$ and $B \in \Qs$. Without loss of
generality, we may assume that $A, B$ and $C$ are monic. If one of the
polynomials, say $C$, is not monic, we write $C=cC'$, where $C' \in \Q[t]$ is
monic and $c$ is the leading coefficient of $C$. By replacing each coefficient
$a_i$ of $A$ (respectively $b_i$ of $B$) by $c^{-1}a_ic$ (respectively
$c^{-1}b_ic$) and again denoting the thus obtained polynomials by $A$ and $B$,
we find a representation $Q=cABC'$ with monic $C' \in \Q[t]$. Similar ideas can
be applied to $A$ and $B$ so that $Q=qA'B'C'$ for an appropriate $q \in \Q$ and
monic polynomials $A',C' \in \Q[t]$, $B' \in \Q[s]$. Since $Q$ is monic, we
conclude $q=1$. Factorizing $A$ and $C$ according to
Theorem \ref{thm:univ}, Part~(c) yields the desired result.
\end{proof}

\begin{rmk}
  In our short proof of Corollary~\ref{univfact} we appeal to \cite{Skopenkov19}
  and known factorization results as illustrated in Section~2.1. We would like
  to mention that already the proof of Theorem~\ref{SK2} in \cite{Skopenkov19}
  is constructive and inductively produces linear univariate left/right factors
  of~$Q$.
\end{rmk}

\section{Equivalence of Factorizations}
\label{sec:equivalence}

So far we have considered existence of factorizations. Before turning to their uniqueness or non-uniqueness we develop a sensible concept of equivalence of factorizations. An obvious source of non-uniqueness of factorizations of $Q \in \Hstar{1}$ is the potential non-uniqueness of factorizations of \emph{univariate} factors of $Q$ (c.~f. Theorem \ref{thm:univ}, Part~(c)). It seems natural to consider two factorizations obtained in this way as equivalent. Definition~\ref{def:equivalence} below provides us with a criterion for this equivalence relation which will be needed later. By $\fact(Q)$ we denote the set of all possible factorizations of $Q$ into univariate linear factors. For better readability let us introduce the following notation for elements of $\fact(Q)$:
\begin{equation*}
	\F{h}{h}{k}{n}{m} \coloneqq (t-h_1)\cdots(t-h_n)(s-h)(t-k_1)\cdots(t-k_m) \in \fact(Q)
\end{equation*}
with $h_i$, $h$, $k_j \in \Q$, $i=1,\ldots,n$, $j = 1, \ldots, m$. Note that
$\F{h}{h}{k}{n}{m}$ refers to a factorized representation of a polynomial, not
the polynomial itself. Formally, one can think of $\F{h}{h}{k}{n}{m}$ as a
$(n+1+m)$-tuple of linear polynomial factors. By virtue of the usual convention
that the value of an empty product equals one, we also write $\F{h}{h}{k}{0}{m}$
and $\F{h}{h}{k}{n}{0}$ for factorizations of the form
$(s-h)(t-k_1)\cdots(t-k_m)$ and $(t-h_1)\cdots(t-h_n)(s-h)$.

\begin{defn}
  \label{def:equivalence}
  For a monic star-one-polynomial $Q \in \Hstar{1}$ with $\gcd(Q)=1$ we consider
  the equivalence relation
\begin{equation}
\label{equiv}
\F{h}{h}{k}{n}{m} \sim \F{\ov{h}}{\ov{h}}{\ov{k}}{l}{r} :\!\iff \prod_{i=1}^{n}\No{(t-h_i)}=\prod_{i=1}^{l}\No{(t-\ov{h}_i)}
\end{equation}
on $\fact(Q)$. By $\Fact(Q) \coloneqq \fact(Q)/\sim$ we denote the corresponding
quotient set.
\end{defn}

\begin{rmk}
  Note that \eqref{equiv} implies
  \begin{equation*}
    \prod_{i=1}^{m}\No{(t-k_i)}=\prod_{i=1}^{r}\No{(t-\ov{k}_i)}.
  \end{equation*}
  Thus,
  Definition~\ref{def:equivalence} is actually symmetric in the factors to the
  left and to the right of the $s$-factor. In case of $n = 0$ or $l = 0$, the
  empty product convention applies.
\end{rmk}

As already mentioned, the equivalence relation of
Definition~\ref{def:equivalence} aims at identifying factorizations which arise
from different factorizations of the univariate polynomials
$(t-h_1)\cdots(t-h_n) \in \Qt$ and $(t-k_1)\cdots(t-k_m) \in \Qt$. This needs a
justification:

\begin{prop}
\label{eindeutigequiv}
Let $Q \in \Hstar{1}$ be a monic star-one-polynomial with $\gcd(Q) = 1$ and
consider two representatives $\F{h}{h}{k}{n}{m}$,
$\F{\ov{h}}{\ov{h}}{\ov{k}}{l}{r}$ of the same equivalence class
$[\F{h}{h}{k}{n}{m}]=[\F{\ov{h}}{\ov{h}}{\ov{k}}{l}{r}] \in \Fact(Q)$. Then
$n=l$, $m=r$, $h = \ov{h}$, $(t-h_1)\cdots(t-h_n) =
(t-\ov{h}_1)\cdots(t-\ov{h}_l)$ and $(t-k_1)\cdots(t-k_m) =
(t-\ov{k}_1)\cdots(t-\ov{k}_r)$, that is, the two factorizations arise from
different factorizations of \emph{univariate} polynomials.
\end{prop}

\begin{proof}
  It is clear that $l=n$ and $r=m$. Write
\begin{align*}
	Q = \ &\F{h}{h}{k}{n}{m}=\underbrace{(t-h_1)\cdots(t-h_n)}_{=:P_1}(s-h)\underbrace{(t-k_1)\cdots(t-k_m)}_{=:P_2}\\
	= \ &\F{\ov{h}}{\ov{h}}{\ov{k}}{n}{m}=\underbrace{(t-\ov{h}_1)\cdots(t-\ov{h}_n)}_{=:\ov{P}_1}(s-\ov{h})\underbrace{(t-\ov{k}_1)\cdots(t-\ov{k}_m)}_{=:\ov{P}_2}
\end{align*}
with $\No{P_1}=\No{\ov{P}_1}$ and $\No{P_2}=\No{\ov{P}_2}$. Without loss of
generality, we may assume $\No{(t-\ov{h}_i)}=\No{(t-h_i)}\eqqcolon M_i$ for $i \in
\{1,\ldots,n\}$ and $\No{(t-\ov{k}_j)}=\No{(t-k_j)}\eqqcolon N_j$ for $j \in
\{1,\ldots,m\}$. If that is not the case, we can compute different
factorizations of $\ov{P}_1$ and $\ov{P}_2$ where the factors appear in the
desired order (c.\,f. Theorem \ref{thm:univ}, Part~(c); note that $\gcd(Q) = 1$
implies $\gcd(\ov{P}_1) = \gcd(\ov{P}_2) = 1$).

There exist polynomials $Q_0$, $Q_1 \in \Qt$ with $Q = Q_0 + sQ_1$. We have
\begin{multline*}
  Q_1=(t-h_1)\cdots(t-h_n)(t-k_1)\cdots(t-k_m)\\=(t-\ov{h}_1)\cdots(t-\ov{h}_n)(t-\ov{k}_1)\cdots(t-\ov{k}_m),
\end{multline*}
and hence $t-k_m$ and $t-\ov{k}_m$ are right factors of $Q_1$ as well as of
$N_m$. Such right factors are uniquely determined, that is $k_m = \ov{k}_m$, as
long as $N_m \nmid Q_1$ (c.\,f. \cite[Lemma 3]{hegedus13}). If $N_m \mid Q_1$, we
pass over to
\begin{multline*}
  Q_0 = -(t-h_1) \cdots (t - h_n)h(t-k_1) \cdots (t-k_m) \\
      = -(t-\ov{h}_1) \cdots (t - \ov{h}_n)\ov{h}(t-\ov{k}_1) \cdots (t-\ov{k}_m).
\end{multline*}
If $N_m \mid Q_0$, then $N_m \mid Q$, a contradiction to $\gcd(Q) = 1$. Hence $N_m \nmid Q_0$ and -- by applying the same ideas to $Q_0$ -- we obtain $k_m=\ov{k}_m$.

Now it is possible to cancel out the factor $t-k_m$ from $Q_0$, $Q_1$ and also
from~$Q$ to obtain polynomials $Q'_0$, $Q'_1$, $Q' = Q'_0 + sQ'_1$. Proceeding
inductively with $Q'$ instead of $Q$ we obtain $k_j = \ov{k}_j$ for $j \in
\{1,\ldots,m\}$. A similar argument for left factors then yields $h_l =
\ov{h}_l$ for $l \in \{1,\ldots,n\}$. This also implies $h = \ov{h}$ and the lemma's claim follows.
\end{proof}

\section{Factorizations and Null Lines}
\label{sec:null-lines}
In this section, we assume that the monic polynomial $Q \in \Hstar{1}$ with $\gcd(Q)=1$ admits a factorization, that is
\begin{equation}
\label{factorization}
	Q=\F{h}{h}{k}{n}{m}=(t-h_1)\cdots(t-h_n)(s-h)(t-k_1)\cdots(t-k_m)
\end{equation}
with $h_1,\ldots, h_n$, $h$, $k_1,\ldots,k_m \in \Q$. In the following, we
develop a criterion to decide whether this factorization is essentially unique,
that is, whether there exists only one equivalence class of factorizations. We
call each linear polynomial $t-h_i$, $i=1,\ldots,n$, a \emph{left factor} of
factorization \eqref{factorization} since it arises as a factor on the left-hand
side of the $s$-factor $s-h$. Analogously, each linear polynomial $t-k_i$,
$i=1,\ldots,m$, is called a \emph{right factor} of factorization
\eqref{factorization}.

For $i=1,\ldots,n$ we define
\begin{equation}
\label{M}
	M_i \coloneqq \No{(t-h_i)} \in \Rt.
\end{equation}
Analogously, for $j=1,\ldots,m$ we set
\begin{equation}
\label{N}
	N_j \coloneqq \No{(t-k_j)} \in \Rt.
\end{equation}
It holds that
\begin{equation}
\label{eq:PR}
\No{Q} = PR \text{ with } P=M_1\cdots M_n N_1 \cdots N_m \text{ and } R=\No{(s-h)}.
\end{equation}
Moreover, for $i=1,\ldots,n$ and $j=1,\ldots,m$ the polynomials $M_i$ and $N_j$ are \emph{monic} and \emph{irreducible} in $\R[t]$ and $\deg M_i=2=\deg N_j$.

A polynomial $Q = Q_0 + sQ_1 \in \Hstar{1}$ gives rise to a ruled surface in
$\P(\Cq)$ which is parameterized as\footnote{Our parametric equation
  \eqref{surface} misses the curve $Q_1(t)$ and the ruling spanned by the
  leading coefficients of $Q_0$ and $Q_1$. This defect could be overcome by
  either homogenizing the polynomial and using projective parameters or by
  properly defining evaluation at parameter values $s = \infty$, $t = \infty$.
  For the sake of simplicity of notation we refrain from doing this. This will
  not affect validity of our arguments.}
\begin{equation}
\label{surface}
  \C \times \C \to \P(\Cq)\colon
  (s,t) \mapsto [Q(s,t)] = [Q_0(t) + sQ_1(t)].
\end{equation}
The $s$-parameter lines are the surface rulings. It will turn out that existence of a factorization of $Q$, its essential uniqueness and also the number of left and right factors are related to special rulings on this surface, namely rulings that lie on the \emph{null quadric}:

\begin{defn}
Consider the symmetric bilinear form $b$ defined as
\begin{equation*}
  \Cq \times \Cq \rightarrow \C\colon (h_1,h_2) \mapsto b(h_1,h_2) = \frac{1}{2}(h_1\Cj{h}_2+h_2\Cj{h}_1).
\end{equation*}
The quadric $\N \subseteq \P(\Cq)$ represented by this bilinear form is
called \emph{null quadric}. The \emph{points} of $\N$ are precisely the elements
$[h] \in \P(\Cq)$ with $b(h,h) = \No{h}=0$, that is, elements in $\P(\Cq)$
represented by complex quaternions with zero norm. Lines in $\N$ are called
\emph{null lines}.
\end{defn}

Let us fix a complex number $z \in \C=\R+\mathrm{i}\R$. For $i=1,2$, we view
$Q_i \in \Q[t]$ as an element of $\Cq[t]$ and define the \emph{evaluation}
$Q_i(z) \in \Cq$ by replacing the indeterminate $t$ by $z$. This substitution is
uncritical since $z$ commutes with all elements in $\Cq$.

Consider the projective span $\rul$ which is parameterized by
\begin{equation*}
  \C \to \P(\Cq): s \mapsto [Q_0(z) + sQ_1(z)].
\end{equation*}
The projective span $\rul$ is either a straight line or degenerates to a point
if $[Q_0(z)] = [Q_1(z)]$. Moreover, it may happen that one of the quaternions,
$Q_0(z)$ or $Q_1(z)$, equals $0$. If that is the case, we also use the notation
$\rul$ even though we actually just consider the point $[Q_1(z)]$ or $[Q_0(z)]$.

It will turn out to be advantageous to classify null lines. It is known that
the null quadric $\N$ is covered by two families of straight lines. Elements of
the first family $\mathcal{L}$ are called \emph{left rulings}, elements of the
second family $\mathcal{R}$ are called \emph{right rulings.} Any null line is
either a left ruling or a right ruling of $\N$. For more details we refer to
\cite[Theorem~8.3.2.]{casas-alvero14}. Each point $[h] \in \N$ is incident with
exactly one left ruling $L_{[h]}$ and exactly one right ruling $R_{[h]}$. In
\cite{Siegele20} it is shown that these two straight lines are given by
\begin{equation}
	\label{lrr}
		L_{[h]} \coloneqq \{[q] \in \P(\Cq): q\Cj{h}=0\} \quad \text{ and } \quad R_{[h]} \coloneqq \{[q] \in \P(\Cq): \Cj{h}q = 0\}.
\end{equation}	
The equations $q\Cj{h}=0$ and $\Cj{h}q=0$ are not equivalent since
multiplication of quaternions is not commutative. The two lines $L_{[h]}$ and
$R_{[h]}$ are indeed different.

If $z \in \C$ is a complex zero of the norm polynomial's univariate factor $P
\in \Rt$, we obtain
\begin{equation*}
  \No{(Q_0(z)+sQ_1(z))}=P(z)R = 0
\end{equation*}
and hence the projective span $\rul$ is (at least contained in) a null line.
(Note that $Q_0(z) = Q_1(z) = 0$ is not possible because of the assumption
$\gcd(Q) = 1$.) The following theorem provides a more precise geometric
interpretation for the zeros of $P$ in terms of left and right rulings. Note the
careful formulation ``contained in a left/right ruling''. It leaves open the
possibility that $\rul$ is just a point.

\begin{thm}
\label{factandrul}
Assume that a monic polynomial $Q=Q_0+sQ_1 \in \Hstar{1}$ with $\gcd(Q)=1$ admits a factorization of the form \eqref{factorization}. Moreover, let $M_i$, $i=1,\ldots,n$, $N_j$, $j=1,\ldots,m$, and $P$ be defined as in \eqref{M}, \eqref{N} and \eqref{eq:PR}. Let $z \in \C$ be a complex zero of $P$.
	\begin{enumerate}
		\item[(a)] If there exists $j \in \{1, \ldots, m\}$ such that $N_j(z)=0$, then $\rul$ is contained in a left ruling.
		\item[(b)] If there exists $i \in \{1,\ldots,n\}$ such that $M_i(z)=0$, then $\rul$ is contained in a right ruling.
	\end{enumerate}
\end{thm}

\begin{proof}
	Let us first consider Part~(a). There is nothing to show if $\rul$ is just a
  point. Hence, we can assume that $[Q_0(z)]$ and $[Q_1(z)]$ are two distinct
  points. We want to show that $[Q_1(z)]$ lies on the unique left ruling through
  $[Q_0(z)]$. By the defining condition \eqref{lrr} for left rulings, this is
  equivalent to $Q_1(z)\Cj{Q_0}(z)=0$ (we use the denotation
  $\Cj{Q_0}(z)=\Cj{(Q_0(z))}$).
	
  Because $Q$ admits a factorization of the form \eqref{factorization}, we have
\begin{equation}
\begin{aligned}
\label{Q0Q1}
	Q_0&=- (t-h_1)\cdots(t-h_n)h(t-k_1)\cdots(t-k_m),\\
	Q_1&=(t-h_1)\cdots(t-h_n)(t-k_1)\cdots(t-k_m)
\end{aligned}
\end{equation}
(we already used these representations in the proof of
Proposition~\ref{eindeutigequiv}). Hence
\begin{equation}
\label{Q1Cj(Q0)}
  \begin{aligned}
	Q_1\Cj{Q_0} = \ - \ &(t-h_1)\cdots(t-h_n)(t-k_1)\cdots(t-k_m)\\ &\Cj{(t-k_m)}\cdots \Cj{(t-k_1)}\Cj{h}\Cj{(t-h_n)}\cdots \Cj{(t-h_1)}\\
	= \ - \ & (t-h_1)\cdots (t-h_n)\Cj{h} \Cj{(t-h_n)}\cdots \Cj{(t-h_1)}N_1\cdots N_m.
  \end{aligned}
\end{equation}
Evaluating at $z$ yields $Q_1(z)\Cj{Q_0}(z) = 0$ since $N_j(z) = 0$. Similarly,
one can show that $M_i(z) = 0$ implies $\Cj{Q_0}(z)Q_1(z) = 0$ which proves
Part~(b) of the theorem.
\end{proof}

Theorem \ref{factandrul} can be interpreted as follows: Each single left factor
of factorization \eqref{factorization} gives rise to two conjugate complex
parameter values $z$, $\overline{z} \in \C$ with $\rul$ (resp.
$[Q_0(\overline{z})] \vee [Q_1(\overline{z})]$) being contained in a right
ruling. Similarly, right factors of \eqref{factorization} lead to (points on)
left rulings.

The algebraic criterion for $\rul$ being contained in a left/right ruling is the
vanishing of $Q_1(z)\Cj{Q_0}(z)$ and $\Cj{Q_0}(z)Q_1(z)$, respectively. Denote by
\begin{equation}
  \label{AB}
  A \coloneqq \{z \in \C: Q_1(z)\Cj{Q_0}(z) = 0\}\text{ and }
  B \coloneqq \{z \in \C: \Cj{Q_0}(z)Q_1(z) = 0\}
\end{equation}
the sets of complex zeros of the polynomials $Q_1\Cj{Q_0}$ and $\Cj{Q_0}Q_1$. We
say that the multiplicity $\m{H}{z}$ of $z \in \C$ as a zero of $H \in \Q[t]$
equals $\mu$ if $(t-z)^\mu$ divides $H$ in $\Cq[t]$ and $(t-z)^{\mu+1}$ does
not. Since $(t-z)$ is part of the center of $\Cq[t]$, we need not distinguish
between left- and right-division.

Let us briefly explain why $H(z)=0$ is equivalent to $(t-z)$ dividing $H$ in
$\Cq[t]$: Write $H=H_0+\qi H_1+ \qj H_2 + \qk H_3$, where $H_i \in \R[t]$ for
$i=0,1,2,3$. The fact $H(z)=0$ implies $H_i(z)=0$. Since $H_i \in \R[t]$ is a
real polynomial, we can find $H_i' \in \C[t]$ such that $H_i=H_i'(t-z)$, whence
$H=H'(t-z)$ with $H'=H_0'+\qi H_1' + \qj H_2' + \qk H_3'$. If $(t-z)$ divides
$H$ in $\Cq[t]$, it also divides $H_i$ in $\C[t]$, which implies $H_i(z)=0$ and
hence $H(z)=0$.

Write $\lambda(z) \coloneqq \m{Q_1\Cj{Q_0}}{z}$ and $\varrho(z) \coloneqq
\m{\Cj{Q_0}Q_1}{z}$. We then define the \emph{multiplicity cardinalities}
\begin{equation}
  \label{mult-card}
  \# A \coloneqq \sum_{z \in A} \lambda(z),\quad
  \# B \coloneqq \sum_{z \in B} \varrho(z).
\end{equation}

\begin{rmk}
\label{rmk:lowerbound}
Under the assumptions of Theorem~\ref{factandrul} we define $N \coloneqq N_1\cdots N_m$ and $M \coloneqq M_1\cdots M_n$. From equation \eqref{Q1Cj(Q0)} it follows that
\begin{equation*}
  \lambda(z) \geq \m{N}{z} \text{ for } z \in A
\end{equation*}
and similarly
\begin{equation*}
  \varrho(z) \geq \m{M}{z} \text{ for } z \in B.
\end{equation*}
Moreover, the $\deg N=2m$ complex zeros of $N$ (counted with multiplicities) are elements of $A$ and the $\deg M=2n$ complex zeros of $M$ (counted with multiplicities) are elements of $B$, which shows that the multiplicity cardinalities $\#A$ and $\#B$ are bounded from below by
\begin{equation}
  \label{eq:A2mB2n}
  \# A \ge 2m \quad \text{and} \quad \# B \ge 2n.
\end{equation}
\end{rmk}

The lower bounds \eqref{eq:A2mB2n} need not be strict:

\begin{example}
For $Q = (t-\qi)(s-\qj)(t-\qj)$ we have
\begin{equation*}
  Q_1\Cj{Q_0}=(t^2 + 1)(\qj t^2 - 2\qk t - \qj) \quad \text{and} \quad \Cj{Q_0}Q_1=\qj(t^2 + 1)^2.
\end{equation*}
Therefore, $m = n = 1$, $A = B =
  \{\pm\mathrm{i}\}$, $\lambda(\pm \mathrm{i}) = 1$, $\varrho(\pm \mathrm{i}) =
  2$ and hence $\#A = 2 = 2m$ but $\#B = 4 > 2n$.
\end{example}

It is natural to relate the algebraic multiplicities $\lambda(z)$ and
$\varrho(z)$ to multiplicities of left/right rulings on the ruled surface $Q$.

Via the Klein mapping $\gamma$ (c.\,f. \cite[Section~2.1]{Pottmann10}), straight
lines in $\P^3$ are mapped to points of the \emph{Plücker quadric} in $\P^5$.
Ruled surfaces are mapped to curves such that the intersection multiplicity of
ruled surfaces at lines can be based on the concept of intersection multiplicity
of curves from projective differential geometry. If $\rul$ is a straight line,
then $\gamma(\rul)$ is a point on the Plücker quadric. Moreover, the Klein images $\gamma(\mathcal{L})$, $\gamma(\mathcal{R})$ of
left/right rulings are conics on the Plücker quadric. Provided $\rul$ is a
straight line, we can therefore consider the intersection of the rational curve
$\gamma(Q)$ with $\gamma(\mathcal{L})$ (or with $\gamma(\mathcal{R})$) at the
point $\gamma(\rul)$ and compute its \emph{intersection multiplicity}. One can show that $\gamma(Q)$ intersects $\gamma(\mathcal{L})$ with intersection multiplicity $\mu$ in $\rul$ if and only if $(t-z)^\mu$ is a factor of $Q_1\Cj{Q_0}$, that is, $\mu = \lambda(z)$ (c.~f. \cite[Proof~of~Theorem~3]{Siegele20}). If $\rul$ is just a point, one can still compute the algebraic multiplicities $\lambda(z)$ and $\varrho(z)$, but the geometric interpretation of multiplicities in terms of left/right rulings is difficult to sustain.

\section{Uniqueness of Factorizations}
\label{sec:uniqueness}

With the help of the multiplicity cardinalities $\#A$ and $\#B$ we are now able to state a condition which guarantees essential
uniqueness of a factorization of $Q$.

\begin{thm}[Uniqueness Theorem]
  \label{unique-alg}
  Let $Q=Q_0+sQ_1 \in \Hstar{1}$, $\gcd(Q)=1$ and $\No{Q} = PR$ with $P \in \R[t]$ and $R \in \R[s]$. Moreover, let $\#A$ and $\#B$ be the multiplicity cardinalities defined in \eqref{mult-card}. Then $\deg P = \#A + \#B$ implies $\vert\Fact(Q)\vert = 1$.
\end{thm}

\begin{proof}
  The assumptions guarantee existence of one factorization $\F{h}{h}{k}{n}{m}$ of $Q$ of shape \eqref{factorization}. Hence $\deg P= 2(m+n)$ and the equality $\deg P=\# A+\# B$ implies $\# A=2m$ and $\# B=2n$. (Note that $\# A \ge 2m$ and $\# B \ge 2n$ is always satisfied by Remark~\ref{rmk:lowerbound}.)

Suppose there exists a second factorization
\begin{equation}
  \label{factorization3}
	Q=\F{\ov{h}}{\ov{h}}{\ov{k}}{l}{r}=(t-\ov{h}_1)\cdots(t-\ov{h}_l)(s-\ov{h})(t-\ov{k}_1)\cdots(t-\ov{k}_r)
\end{equation}
such that $[\F{h}{h}{k}{n}{m}] \neq [\F{\ov{h}}{\ov{h}}{\ov{k}}{l}{r}]$. By definition of (non)-equivalence,
\begin{equation}
\label{irred1}
	M \coloneqq \prod_{i=1}^n \No{(t-h_i)} \neq \prod_{i=1}^l \No{(t-\ov{h}_i)} \eqqcolon \ov{M}
\end{equation}
and also
\begin{equation}
\label{irred2}
  N \coloneqq \prod_{i=1}^{m}\No{(t-k_i)} \neq \prod_{i=1}^{r}\No{(t-\ov{k}_i)} \eqqcolon \ov{N}
\end{equation}
(we already defined the polynomials $M$ and $N$ in Remark \ref{rmk:lowerbound}).
From \eqref{irred1} and \eqref{irred2} one already obtains the decompositions of
$M$, $\ov{M}$, $N$ and $\ov{N}$ into monic irreducible real factors. There are
two possible cases which ought to lead to the inequality $M \neq \ov{M}$:

\emph{Case 1:} There exists an irreducible factor $\ov{M}_i \coloneqq
\No{(t-\ov{h}_i)}$ of $\ov{M}$ which occurs with higher power in the irreducible
factor decomposition of $\ov{M}$ than in the one of $M$. More precisely, if $z
\in \C$ is a zero of $\ov{M}_i$, we have $\m{M}{z}<\m{\ov{M}}{z}$.

Let us again use representation \eqref{Q0Q1} for the polynomials $Q_0$ and
$Q_1$. Then we have $\Cj{Q_0}Q_1 = HM$ where
\begin{equation}
\label{H}
  H = -\Cj{(t-k_m)}\cdots \Cj{(t-k_1)}\Cj{h}(t-k_1)\cdots(t-k_m).
\end{equation}
For each $z \in B$ it holds that $H(z) \neq 0$. Indeed, $H(z)=0$ would imply
$\# B > 2n$ since each of the $\deg M=2n$ complex zeros of $M$ (counted with
multiplicities) is an element of $B$ (c.~f. Remark \ref{rmk:lowerbound}). Therefore, one actually obtains $\varrho(z) = \m{M}{z}$.

Let us cause a contradiction by using factorization \eqref{factorization3}: This factorization yields another description of the polynomials $Q_0$ and $Q_1$:
\begin{equation*}
  \begin{aligned}
    Q_0 &= -(t-\ov{h}_1)\cdots(t-\ov{h}_l)\ov{h}(t-\ov{k}_1)\cdots(t-\ov{k}_r),\\
    Q_1 &= (t-\ov{h}_1)\cdots(t-\ov{h}_l)(t-\ov{k}_1)\cdots(t-\ov{k}_r).
  \end{aligned}
\end{equation*}
We obtain $\Cj{Q_0}Q_1 = \ov{H}\ov{M}$ where
\begin{equation*}
\ov{H}=-\Cj{(t-\ov{k}_r)}\cdots\Cj{(t-\ov{k}_1)}\Cj{\ov{h}}(t-\ov{k}_1)\cdots(t-\ov{k}_r),
\end{equation*}
which implies the contradiction $\varrho(z) \geq
\m{\ov{M}}{z}>\m{M}{z}=\varrho(z)$.

\emph{Case 2:} There exists an irreducible factor $M_i \coloneqq \No{(t-h_i)}$ of $M$ and $z \in \C$ such that $\m{\ov{M}}{z}<\m{M}{z}$. But then $\m{\ov{N}}{z}>\m{N}{z}$ and we can proceed analogously by using the representation $Q_1\Cj{Q_0}=KN$ with
\begin{equation*}
  K = -(t-h_1)\cdots (t-h_n)\Cj{h}\Cj{(t-h_n)}\cdots\Cj{(t-h_1)}
\end{equation*}
and the fact that $\# A=2m$.
\end{proof}

\begin{rmk}
The converse of Theorem \ref{unique-alg} is also true. If
$|\Fact(Q)|=1$, we obtain $\deg(P)=\#A+\#B$. However, this is not yet obvious
and we need further preparations to be able to formulate and prove the statement
(c.f. Theorem~\ref{nonunique-alg} in Section~\ref{sec:non-uniqueness}).
	
\end{rmk}

\begin{cor}
\label{unique-geom}
Suppose that $Q \in \Hstar{1}$ satisfies the assumptions of Theorem \ref{unique-alg}. If $\rul$ is a straight line for all zeros $z \in \C$ of the norm
polynomial's univariate $t$-factor $P \in \R[t]$, then $\vert\Fact(Q)\vert=1$.
\end{cor}

\begin{proof}
  Since $\rul$ is a straight line for all zeros $z$ of $P$, the sets $A$ and $B$ are disjoint. Therefore, each zero of $P$ is \emph{either} contained in $A$ \emph{or} in $B$. More precisely,
  \begin{equation*}
    \deg P=2(m+n)=\# A+ \#B
  \end{equation*}
  and the statement follows from Theorem~\ref{unique-alg}.
\end{proof}

Theorem \ref{unique-alg} provides us with a sufficient test whether a
factorization of $Q$ is essentially unique. We can compute the multiplicity
cardinalities $\#A$ and $\#B$ of the sets $A$ and $B$, respectively, whence
$\deg P = \#A + \#B$ implies $\vert\Fact(Q)\vert = 1$.

The elements of $A$ and $B$ are the complex numbers $z \in \C$ determined by
\eqref{AB}, the multiplicity cardinalities are the sums of multiplicities
$\lambda(z) = \m{Q_1\Cj{Q_0}}{z}$ and $\varrho(z) = \m{\Cj{Q_0}Q_1}{z}$ for $z
\in A$ and $z \in B$, respectively. We illustrate this at hand of a couple of
examples. The first example demonstrates that the converse of
Corollary~\ref{unique-geom} is not true.

\begin{example}
\label{ex4}
  For
  \begin{equation*}
    Q = (-\qi t - \qj - \qk t  + t^2)s + \qi t - \qj t^2 + \qk t + 1
  \end{equation*}
  we have $A = B = \{\pm \mathrm{i} \}$ and $\lambda(\pm
  \mathrm{i}) = \varrho(\pm \mathrm{i}) = 1$. Hence $\deg P=4=\#A+\#B$ where $P \in \Rt$ again denotes the norm polynomial's univariate $t$-factor. The factorization is essentially unique by Theorem~\ref{unique-alg}. It is
  given by $Q = (t - \qi) (s - \qj) (t - \qk)$ (and trivially unique because we
  only have one left factor and one right factor). However,
  \begin{equation*}
    Q_0(\pm\mathrm{i}) = 1 \pm \qi\mathrm{i} + \qj \pm \qk\mathrm{i} = -Q_1(\pm\mathrm{i})
  \end{equation*}
  so that $[Q_0(\pm\mathrm{i})] \vee [Q_1(\pm\mathrm{i})]$ is just a point.
\end{example}

\begin{example}
The norm of the polynomial
\begin{equation*}
  \begin{aligned}
	Q \coloneqq \ &(\qi(2t^3 + t^2 - 112t + 90) + \qj(t^3 + 11t^2 - 7t - 156)\\
	 + \ &\qk(-4t^3 + 36t^2 - 87t + 18) + t^4 - 9t^3 + 37t^2 - 80t)s\\
	 + \ &\qi(-9t^4 + 88t^3 - 573t^2 + 1136t + 1260)\\
   + \ &\qj(7t^4 - 52t^3 + 217t^2 - 1310t + 1092)\\
	 + \ &\qk(-3t^4 + 32t^3 + 49t^2 - 1100t + 1344) - t^4 + 8t^3 - 23t^2 - 258t
  \end{aligned}
\end{equation*}
equals
\begin{equation*}
	\No{Q} = \underbrace{(t^2 + 2)(t^2 - 2t + 15)(t^2 - 10t + 28)(t^2 - 6t + 39)}_{=:P}\underbrace{(s^2 - 2s + 140)}_{=:R},
\end{equation*}
which shows that the necessary factorization condition is satisfied. Hence $Q$
admits a factorization into univariate linear factors. By computing $A$ and $B$
as defined in \eqref{AB}, we see that $\#A = 4$ and $\#B = 4$. In total, $\deg P = 8 = \#A+\#B$, which shows that the
factorization is essentially unique and the position of the $s$-factor is
unique. Moreover, the fact $\#A=4$ and $\#B=4$ implies that any factorization of $Q$ is of the form
\begin{equation*}
	Q=(t-h_1)(t-h_2)(s-h)(t-k_1)(t-k_2)
\end{equation*}
for appropriate quaternions $h_1$, $h_2$, $h$, $k_1$, $k_2 \in \Q$ and all
factorizations of $Q$ are obtained from all factorizations of the univariate
polynomials $(t-h_1)(t-h_2)$ and $(t-k_1)(t-k_2)$.
\end{example}

\begin{example}
\label{examplemult}
Consider the polynomial
\begin{equation*}
  Q \coloneqq (-2\qi t^2 - \qk(t^2 + 1) + t^3 - t)s + \qi(t^2 - 1) - \qj(t^3 + t) + 2t.
\end{equation*}
Its norm polynomial factors as
\begin{equation*}
  \No{Q} = PR = (t^2 + 1)^3(s^2 + 1),
\end{equation*}
whence $\deg P=6$. Moreover, $A=B=\{\pm \mathrm{i}\}$ and $\#A = 4$,
$\#B=2$. We therefore obtain a factorization with one left factor and two right factors.
The two right factors share the same norm polynomial, whence the factorization
is even unique and not just essentially unique (c.~f. Section
\ref{sec:univariate}). It is given by
\begin{equation*}
  Q = (t-\qi)(s-\qj)(t-\qk)(t-\qi).
\end{equation*}
\end{example}

\begin{example}
\label{ex2}
The polynomial
\begin{multline*}
		Q \coloneqq (\qi (2-t) + \qj (2t^2 - 6t + 5) - \qk t + t^3 - 4t^2 + 5t - 1)s\\
		+ \qi (-t^3 + 4t^2 - 8t + 5) + \qj (-t^3 + 4t^2 - 4t - 1) + \qk (t - 3) + 2t^2 - 7t + 5
\end{multline*}
satisfies the necessary factorization condition
\begin{equation*}
	\No{Q} = PR = (t^2 - 2t + 2)(t^2 - 2t + 3)(t^2 - 4t + 5)(s^2 + 2).
\end{equation*}
The multiplicity cardinalities are $\#A=4$, $\#B=4$ whence $\deg P =6
<\#A+\#B$. Indeed, we will see that this property implies existence of
non-equivalent factorizations (c.~f. Example~\ref{ex2b}).
\end{example}

\section{Non-Uniqueness of Factorizations}
\label{sec:non-uniqueness}

A polynomial $Q \in \Hstar{1}$ admits a factorization if it satisfies the
factorization condition of Definition~\ref{nec}. By Theorem~\ref{unique-alg} the
factorization is in general essentially unique, that is, the set $\Fact(Q)$ of
equivalence classes consists of only one element. However, exceptions do exist.
We study this in the simple case of $Q \in \Q_{11}$.

\begin{example}
  Consider the polynomial $Q = (t-h)(s-k) \in \Q_{11}$ with $h$, $k \in \Q$. A
  second factorization is necessarily of the shape $Q = (s-\ov{k})(t-\ov{h})$.
  Equating and expanding gives
	\begin{equation*}
		(t-h)(s-k) = t\,s - h\,s - k\,t + h\,k = t\,s - \ov{h}\,s - \ov{k}\,t + \ov{k}\,\ov{h} = (s-\ov{k})(t-\ov{h}).
	\end{equation*}
  Now we compare coefficients and find $h = \ov{h}$, $k = \ov{k}$ and $hk = \ov{k}\ov{h} = kh$.
  Thus, a second factorization exists if and only if $h$ and $k$ commute.
\end{example}

Building on this example we see that if a polynomial $Q \in \Hstar{1}$ admits the factorization
\begin{equation*}
  Q = (t-h_1)\cdots(t-h_n)(s-h)(t-k_1)\cdots(t-k_m),
\end{equation*}
then commutativity of $h_n$ and $h$ (or $h$ and $k_1$) implies existence of a
second, non-equivalent, factorization. However, one may find examples where
non-equivalent factorizations do not arise in this simple way, at least not obviously.

\begin{example}
\label{ex2b}
We again consider the polynomial $Q$ of Example~\ref{ex2}. It satisfies the
necessary factorization condition. By Corollary \ref{univfact}, $Q$ completely
decomposes into univariate linear factors, for instance:
\begin{align}
\label{fact1}
	Q=(t - 2 - \qi)(s-\qi-\qj)\left(t + \frac{4\qi}{5} + \frac{3\qj}{5} - 1\right)\left(t + \frac{\qi}{5} + \frac{7\qj}{5} - 1\right)
\end{align}
Another factorization of $Q$ is given by
\begin{align}
\label{fact2}
	Q=(t+\qj+\qk-1)(t-\qk-2)(s-\qi-\qj)(t+\qj-1).
\end{align}
Still, $s-\qi-\qj$ does not commute with any $t$-factor in \eqref{fact1} or
\eqref{fact2}. Obviously, the two factorizations are not equivalent.
\end{example}

As we have seen in Section \ref{sec:uniqueness}, non-uniqueness of
factorizations of $Q$ is only possible if $\deg P < \#A + \#B$ where $P$ denotes the norm polynomial's univariate $t$-factor and $\#A$ and $\#B$ are the multiplicity cardinalities. It turns out
that this necessary condition for existence of non-equivalent factorizations is
also sufficient.

\begin{thm}[Non-Uniqueness Theorem]
\label{nonunique-alg}
Let $Q=Q_0+sQ_1 \in \Hstar{1}$, $\gcd(Q)=1$ and $\No{Q} = PR$ with $P \in \R[t]$ and $R \in \R[s]$. Moreover, let $\#A$ and $\#B$ be the multiplicity cardinalities defined in \eqref{mult-card}. Then the following statements are equivalent:
\begin{enumerate}
  \item[(a)] $\vert\Fact(Q)\vert>1$.
  \item[(b)] $\deg P < \#A + \#B$.
\end{enumerate}
\end{thm}

\begin{proof}
  The assumptions guarantee existence of one factorization $\F{h}{h}{k}{n}{m}$
  of $Q$ of shape \eqref{factorization}. By Theorem~\ref{unique-alg}, (a)
  implies $\deg P \neq \#A + \#B$. Statement (b) then follows from $\deg
  P=2(m+n) \leq \#A+\#B$.

Let us prove that (b) implies (a). From $\deg P = 2(m+n) < \#A + \#B$ we infer $\#A >
2m$ or $\#B > 2n$. Let us assume $\#B > 2n$. The elements of $B$ are precisely the complex zeros of
$\Cj{Q_0}Q_1$. Let us again use the representation $\Cj{Q_0}Q_1=HM$ with
\begin{equation*}
  H = -\Cj{(t-k_m)}\cdots \Cj{(t-k_1)}\Cj{h}(t-k_1)\cdots(t-k_m) \quad \text{and} \quad M = \prod_{i=1}^{n}M_i,
\end{equation*}
where $M_i \coloneqq \No{(t-h_i)}$. Since $\deg M=2n$, $\#B>2n$ implies
existence of at least one $z \in B$ satisfying $H(z)=0$. But then we also obtain
$H(\overline{z})=0$, where $\overline{z}$ denotes the complex conjugate of $z$.
This is due to the fact that $H(z)=0$ is equivalent to $H_i(z)=0$ for
$i=0,1,2,3$, where $H=H_0+\qi H_1 + \qj H_2 + \qk H_3$ (we already used this
representation in the paragraph following equation \eqref{AB}). However, $H_i
\in \R[t]$ is a real polynomial, therefore $H_i(z)=0$ also implies
$H_i(\overline{z})=0$. As a consequence, we obtain $H(\overline{z})=0$. Hence
\begin{equation}
\label{vf}
  H = -\Cj{(t-k_m)}\cdots \Cj{(t-k_1)} \Cj{h} (t-k_1) \cdots (t-k_m) = N^{l}K,
\end{equation}
where $l = \m{H}{z}$, $l \ge 1$, $N \coloneqq (t-z)(t-\overline{z}) \in \Rt$ and $K \in \Qt$ with $N \nmid K$.

The class $[\F{h}{h}{k}{n}{m}]$ of equivalent factorizations contains one element where $N$ is the norm polynomial of the leftmost right factors (the $t$-factors
immediately following the $s$-factor). More precisely, there exists $r \in \{1,\ldots,m\}$ such that $\No{(t-k_i)}=N$ for $i =1, \ldots,r$ and $\No{(t-k_j)}\neq N$ for $j = r+1,\ldots,m$. It is no loss of generality to assume
that this is the given factorization $\F{h}{h}{k}{n}{m}$ of~$Q$.

Consider the product
\begin{equation}
\label{samenorm}
(t-k_1)\cdots(t-k_r).
\end{equation}
This factorization into univariate linear factors is unique because all linear
factors have the same norm polynomial (c.\,f. Theorem \ref{thm:univ}, Part~(c)).
From \eqref{vf} we obtain $N^lK = EF$ where $E \coloneqq -\Cj{(t-k_m)}\cdots
\Cj{(t-k_1)} \Cj{h}$ and $F \coloneqq (t-k_1)\cdots (t-k_m)$. By shifting
$\Cj{h}$ to the left-hand side of the factorization, we get $E =
-\Cj{h}(t-\Cj{(hk_mh^{-1})})\cdots (t-\Cj{(hk_1h^{-1})})$. Neither $E$ nor $F$
have a real polynomial factor of positive degree. Indeed, existence of such a
factor would imply $\gcd(Q)\neq 1$. However, $N^l$ is a factor of $EF$. By
\cite[Proposition~2.1]{Cheng16}, this is only possible if $l$ linear right
factors of $E$ are conjugate to $l$ linear left factors of $F$ and $N^l$ is the
product of these factors. Due to the uniqueness of factorization
\eqref{samenorm}, we obtain
\begin{equation*}
  t-k_i = t - hk_ih^{-1}
\end{equation*}
for $i \in \{1,\ldots,l\}$. Therefore,
\begin{equation*}
  \forall i \in \{1, \ldots, l\}: hk_i=k_ih.
\end{equation*}
Hence $s-h$ commutes with $t-k_1$, \ldots, $t-k_l$. By letting some of these $t$-factors commute with the $s$-factor, one obtains another non-equivalent factorization.

If $\#A>2m$, we can argue similarly and see that $s-h$ commutes with at least
one left factor.
\end{proof}

According to Theorem \ref{nonunique-alg}, $\vert\Fact(Q)\vert > 1$ implies that each
equivalence class in $\Fact(Q)$ can be represented by a factorization
$\F{h}{h}{k}{n}{m}$ where either $t-h_n$ and $s-h$ commute or $s-h$ and $t-k_1$
commute. If $t-h_n$ and $s-h$ commute, then $\F{h}{h}{\ov{k}}{n-1}{m+1}$ with
$\ov{k}_1 = h_n$ and $\ov{k}_l = k_{l-1}$ for $l \in \{2,\ldots,m+1\}$
is a new (non-equivalent) factorization. We call the transition from
$[\F{h}{h}{k}{n}{m}]$ to $[\F{h}{h}{\ov{k}}{n-1}{m+1}]$ a \emph{left jump} of
the $s$-factor. Similar statements hold true for commuting factors $s-h$ and $t
- k_1$. In this case we speak of a \emph{right jump} of the $s$-factor.
Let us demonstrate the statements of Theorem~\ref{nonunique-alg} by means of an example.

\begin{example}
\label{exmult}
Consider the polynomial
\begin{equation}
\label{factmultex}
  Q = \left(t+\qk\right)\left(t-\frac{2\qi}{3}+\frac{2\qj}{3}-\frac{\qk}{3}\right) \left(t-\frac{\qi}{3}+\frac{\qj}{3}-\frac{5\qk}{3}\right) \left(s-2\qk\right) \left(t-\qi-\qj+\qk\right).
\end{equation}
Its norm polynomial is given by
\begin{equation*}
  \No{Q} = PR
  \quad\text{with}\quad
  P = (t^2 + 3)^2(t^2 + 1)^2,\quad R = s^2 + 4.
\end{equation*}
From \eqref{factmultex} we obtain $n=3$ and $m=1$. Moreover, $\#A=6>2m$,
$\#B=6=2n$ and $\deg P < \#A + \#B$. Hence, by Theorem \ref{nonunique-alg},
$\vert\Fact(Q)\vert>1$. Let us precisely investigate the elements of set $A$. It holds
that
\begin{equation*}
  A = \{\mathrm{i},-\mathrm{i},\sqrt{3}\mathrm{i},-\sqrt{3}\mathrm{i}\}.
\end{equation*}
The right factor $t-\qi-\qj+\qk$ corresponds to the complex numbers
$\sqrt{3}\mathrm{i}$ and $-\sqrt{3}\mathrm{i}$ (we have
$\No{(t-\qi-\qj+\qk)}=t^2+3$). Moreover, $\lambda(\pm \sqrt{3}\mathrm{i})=1$.
Consequently, $s-2\qk$ does not perform a \emph{left jump} by commuting with a
left factor of norm $t^2+3$.

Let us now consider the elements $\mathrm{i},-\mathrm{i} \in A$. It holds that $\lambda(\pm \mathrm{i})=2$. There does not exist a right factor of factorization \eqref{factmultex} with norm polynomial $t^2+1$. Therefore, two left jumps of the $s$-factor are possible. Following the proof of Theorem \ref{nonunique-alg}, we compute another (equivalent) factorization of $Q$ where the two rightmost left factors have norm polynomial $t^2+1$. The respective factors commute with $s-2\qk$:
\begin{equation*}
  Q = (t-\qi+\qj+\qk)(t-\qk)(t-\qk)(s-2\qk)(t-\qi-\qj+\qk)
\end{equation*}
Indeed, $(t-\qk)(s-2\qk)=(s-2\qk)(t-\qk)$. Further equivalence classes of factorizations of $Q$ can be found in this way, that is
\begin{gather*}
    [(t-\qi+\qj+\qk)(t-\qk)(t-\qk)(s-2\qk)(t-\qi-\qj+\qk)],\\
    [(t-\qi+\qj+\qk)(t-\qk)(s-2\qk)(t-\qk)(t-\qi-\qj+\qk)],\\
    [(t-\qi+\qj+\qk)(s-2\qk)(t-\qk)(t-\qk)(t-\qi-\qj+\qk)].
\end{gather*}
By Theorem~\ref{jumps} below, these are all elements of $\Fact(Q)$.
\end{example}

In Corollary \ref{unique-geom} we provided a geometric condition which
guarantees uniqueness of a factorization of $Q$: If $\rul$ is a straight line
for all zeros $z \in \C$ of the norm polynomial's univariate factor $P \in \Rt$,
we obtain $\vert\Fact(Q)\vert=1$. However, as shown in Example \ref{ex4}, the converse
need not be true. Nevertheless, it turns out to be true if we require an
additional assumption to be satisfied:

\begin{cor}
Let $Q \in \Hstar{1}$ satisfy the assumptions of Theorem~\ref{nonunique-alg}. Define the polynomials $M$ and $N$ according to $\eqref{irred1}$ and $\eqref{irred2}$ and suppose that $\gcd(M,N)=1$. If there exists a zero $z \in \C$ of the norm polynomial's univariate $t$-factor $P \in \Rt$ such that $\rul$ is just a point, we obtain $\vert\Fact(Q)\vert>1$.
\end{cor}

\begin{proof}
  Let us first assume that $z$ is a zero of $N$ (note that $P=MN$). Since $\rul$ is just a point, we obtain $\Cj{Q_0}(z)Q_1(z)=H(z)M(z)=0$, where $H$ is defined according to \eqref{H}. The fact $\gcd(M,N)=1$ then implies $H(z)=0$. Following the proof of Theorem \ref{nonunique-alg}, a \emph{right jump} of the $s$-factor is possible. If $z$ is a zero of $M$, one can perform a \emph{left jump} of the $s$-factor. In both cases one obtains $\vert\Fact(Q)\vert>1$.
\end{proof}

\begin{thm}
  \label{jumps}
  Suppose $Q \in \Hstar{1}$ satisfies the assumptions of Theorem \ref{nonunique-alg}. All elements of $\Fact(Q)$ can be obtained by repeated application of left and
  right jumps of the $s$-factor.
\end{thm}

\begin{proof}
  Given two different classes $[\F{h}{h}{k}{n}{m}]$ and
  $[\F{\ov{h}}{\ov{h}}{\ov{k}}{l}{r}]$ of factorizations we proceed as follows.
  We define $M$ and $\ov{M}$ according to \eqref{irred1}. For each complex zero
  $z \in \C$ of the norm polynomial's univariate factor $P \in \Rt$ we compute
  the multiplicities $\m{M}{z}$ and $\m{\ov{M}}{z}$. If
  $\m{M}{z}<\m{\ov{M}}{z}$, we follow the proof of Theorem \ref{nonunique-alg}
  and successively perform \emph{right jumps} of the $s$-factor of factorization
  $[\F{h}{h}{k}{n}{m}]$ until the multiplicities coincide. Similarly, if
  $\m{\ov{M}}{z}<\m{M}{z}$, we obtain equality of multiplicities by sequentially
  applying \emph{left jumps} of the $s$-factor of factorization
  $[\F{h}{h}{k}{n}{m}]$. We then obtain an equivalence class which is equal to
  $[\F{\ov{h}}{\ov{h}}{\ov{k}}{l}{r}]$ by Definition~\ref{def:equivalence}.
\end{proof}

\begin{example}
  \label{ex2c}
  Let us illustrate Theorem~\ref{jumps} at hand of the polynomial $Q$ of
  Examples~\ref{ex2} and \ref{ex2b}. We have already computed the two
  non-equivalent factorizations
  \begin{equation*}
    \begin{aligned}
      Q &= (t - \qi -2)(s-\qi-\qj)\Bigl(t + \frac{4\qi}{5} + \frac{3\qj}{5} - 1\Bigr)\Bigl(t + \frac{\qi}{5} + \frac{7\qj}{5} - 1\Bigr) \\
      &= (t+\qj+\qk-1)(t-\qk-2)(s-\qi-\qj)(t + \qj - 1).
    \end{aligned}
  \end{equation*}
  Moreover, we have
  \begin{equation*}
    \begin{aligned}
    t^2-4t+5&=\No{(t-\qi-2)}\\ &= \No{(t-\qk-2)},\\
    t^2-2t+2&=\No{\Bigl(t + \frac{4\qi}{5} + \frac{3\qj}{5} - 1\Bigr)}\\ &= \No{(t+\qj-1)},\\
    t^2-2t+3&=\No{\Bigl(t+\frac{\qi}{5}+\frac{7\qj}{5}-1\Bigr)} \\
    & = \No{(t+\qj+\qk-1)}.
    \end{aligned}
  \end{equation*}
  The first quadratic polynomial corresponds to a left factor and the second to
  a right factor in both factorizations, respectively. In order to make the two
  factorizations equal (or equivalent) by jumps of the $s$-factor, we should
  therefore consider the third quadratic polynomial.

  The product of the two rightmost $t$-factors of the first factorization admits
  a second factorization:
\begin{equation*}
	\Bigl(t + \frac{4\qi}{5} + \frac{3\qj}{5} - 1\Bigr)\Bigl(t + \frac{\qi}{5} + \frac{7\qj}{5} - 1\Bigr) = (t + \qi + \qj - 1) (t + \qj - 1).
\end{equation*}
The factors $s-\qi-\qj$ and $t+\qi+\qj-1$ commute so that
\begin{equation*}
	Q=(t - 2 - \qi)(t + \qi + \qj - 1)(s-\qi-\qj)(t + \qj - 1).
\end{equation*}
This is already equivalent to the second factorization since
$(t - \qi - 2)(t + \qi + \qj - 1) = (t+\qj+\qk-1)(t-\qk-2)$.
\end{example}

\begin{rmk}
\label{rmk:possjumps}
Suppose $Q\in \Hstar{1}$ satisfies the assumptions of Theorem~\ref{nonunique-alg}. The number of possible jumps of the $s$-factor of a given factorization of $Q$ can be counted with the help of the multiplicity cardinalities $\#A$ and $\#B$. One can perform
\begin{equation*}
	\#A+\#B-\deg P
\end{equation*}
jumps of the $s$-factor. More precisely, $\#A-2m$ left jumps and $\#B-2n$ right jumps are possible, where $m$ denotes the number of right factors and $n$ denotes the number of left factors of the given factorization of $Q$. This immediately follows from the proof of Theorem~\ref{nonunique-alg}.
\end{rmk}

\begin{cor}
	Suppose $Q \in \Hstar{1}$ satisfies the assumptions of Theorem \ref{nonunique-alg}. All elements of $\fact(Q)$, that is, all possible factorizations of $Q$ with monic univariate linear factors, can be found by performing the following three steps:
\begin{enumerate}
\item[Step 1:] Compute a factorization of $Q$ with monic univariate linear
  factors (Theorem~\ref{SK2}, Corollary~\ref{univfact}).
\item[Step 2:] Perform all possible left/right jumps of the $s$-factor to obtain
  $\#A+\#B-\deg P$ representatives of different equivalence classes of
  $\Fact(Q)$. All representatives are of the form $A(s-h)B$ with univariate
  polynomials $A, B \in \Qt$ and $h \in \Q$ (Theorem~\ref{jumps} and Remark
  \ref{rmk:possjumps}).
\item[Step 3:] Compute all possible factorizations of $A$ and $B$ with
  univariate linear factors (Theorem~\ref{thm:univ}, Part~(c)).
\end{enumerate}
\end{cor}

\begin{proof}
	The statement is just a summary of the mentioned theorems and corollaries.
\end{proof}

\section{Applications in Kinematics and Future Research}
\label{sec:kinematics}

As mentioned in Section~\ref{sec:introduction}, one motivation for our study are
applications in kinematics. This section explains the underlying ideas and
demonstrates, why the factorization theory of bivariate quaternionic polynomials
needs to be extended to polynomials of higher bi-degree in order to allow the
construction of interesting mechanisms. For background information in the
relation of quaternions to (spherical) kinematics we refer to \cite{hegedus13}.

A vector $(x_1,x_2,x_3) \in \R^3$ is identified with the \emph{vectorial
  quaternion $x = x_1\qi + x_2\qj + x_3\qk$.} The quaternion $q \in \Q \setminus
\{0\}$ acts on $\R^3$ via
\begin{equation}
  \label{rotation}
  x \mapsto \frac{qx\Cj{q}}{\No{q}}.
\end{equation}
The map \eqref{rotation} is a rotation around the vector $q-\Cj{q}$ (or the
identity if $q - \Cj{q} = 0$). The action \eqref{rotation} can be extended to
quaternionic polynomials by replacing $q$ with a quaternionic polynomial.
Univariate polynomials then give one-parametric rational spherical motions,
bivariate polynomials give two-parametric motions etc. For the sake of
simplicity, we henceforth do not distinguish between motions and polynomials.

In this sense, the linear polynomial $t - h \in \Q[t]$ is a rotation around the
fixed vector $h-\Cj{h}$. This important observation relates our factorizations
with linear univariate factors to mechanisms with revolute joints. A
factorization $\F{h}{h}{k}{n}{m}$ of $Q \in \Hstar{1}$ describes a mechanism
consisting of $n+1+m$ revolute joints, connected in sequential order, that can
perform the motion $Q$. The mechanism's movement when following the motion $Q$
requires that the first $n$ joints and the last $m$ joints share the same motion
parameter. In order to achieve this mechanically, that is, without individually
controlling each joint parameter, it is necessary to further constrain the
mechanism. In the univariate case this can be done by considering further
factorizations, each yielding a new sub-mechanism or ``leg'' that can be added.
In this way, we may think of the univariate $t$-factors $T_l =
(t-h_1)\cdots(t-h_n)$ to the left and $T_r = (t-k_1)\cdots(t-k_m)$ to the right
of the $s$-factor as ``higher-order'' joints with a single degree of freedom.
The ``mechanism'' $T_l(s-h)T_r$ corresponds to all factorizations in one class
of equivalent factorizations and serves as one single leg.

The polynomial $Q$ may admit non-equivalent factorizations but they cannot be
used as additional legs because they do not further constrain the mechanism: The
left and right jumps of Theorem~\ref{jumps} interchange commuting factors. These
have linearly dependent vector parts and hence correspond to identical revolute
joints. Therefore, the legs to non-equivalent factorizations are actually identical
from a mechanism science viewpoint.

This is illustrated at hand of a simple example in Figure~\ref{fig:kinematics}.
This figure refers to planar kinematics (all revolute axes are parallel) while
our results pertain to spherical kinematics (all revolute axes are concurrent).
The principle ideas and problems are the same but planar kinematics is easier
and clearer to visualize. We consider a polynomial $Q = T_l(s-h)(t-k)$ where
$\deg T_l = 2$. The top row illustrates our original hopes: The first image
shows the mechanism to the equivalence class of the factorization
$T_l(s-h)(t-k)$, the second image shows the mechanism to the equivalence class
obtained after a right jump of $s-h$. This mechanism has the desired two degrees
of freedom and could, in principle, be used as one leg. However, our results
imply that $s-h$ and $t-k$ commute so that the true situation is that of the
second row where the left and the right mechanisms are actually the same.

\begin{figure}
  \centering
  \begin{overpic}{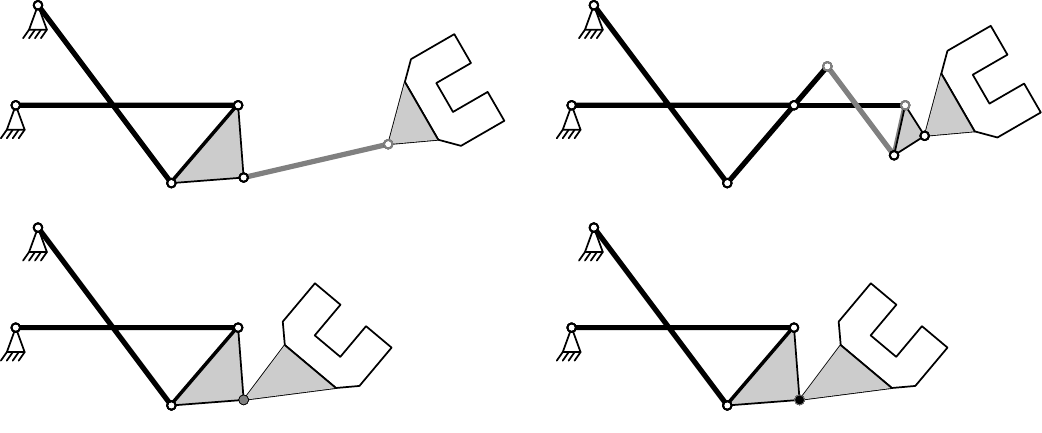}
    \put(13,36){$T_l$}
    \put(20,22){$s-h$}
    \put(35,24){$t-k$}
    \put(70,36){$\hat{T}_l$}
    \put(88,26){$s-h$}
    \put(22,0){$(s-h)(t-k)$}
    \put(77,0){$(t-k)(s-h)$}
    \put(13,14){$T_l$}
    \put(68,14){$T_l$}
  \end{overpic}
  \caption{The construction of mechanisms from non-equivalent factorizations
    fails as the two mechanisms in the bottom row are identical.}
  \label{fig:kinematics}
\end{figure}
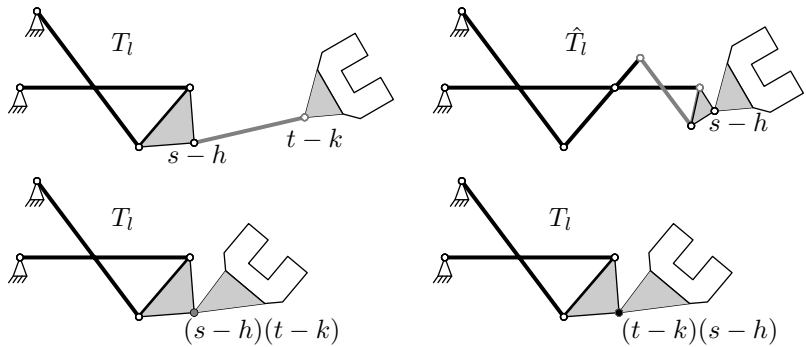

Failure of immediate kinematic applications should not prevent us from further
investigation on factorizability of bi- and multivariate quaternionic
polynomials. A natural question is factorizability of polynomials of arbitrary
bi-degree which is addressed in \cite{lercher21}. That article also presents an
example of a closed-loop mechanism of eight revolute joints with remarkable
properties whose construction is based on two factorizations of a quaternionic
polynomial of bi-degree $(2,2)$ and its extension to the algebra of dual
quaternions \cite{hegedus13,Husty10,Gentili21}. Our investigations in this
article provide necessary foundations for these extensions.

\section*{Acknowledgments}

Daniel F. Scharler was supported by the Austrian Science Fund (FWF): P~31061 The Algebra of Motions in 3-Space.
Johannes Siegele was supported by the Austrian Science Fund (FWF): P~30673 Extended Kinematic Mappings and Application to Motion Design.

\bibliographystyle{elsarticle-num}

\end{document}